\documentclass[12pt,reqno]{article}
mm \textheight=8.9in

\usepackage{amssymb}
\usepackage{amsmath}
\usepackage{amsthm}
\usepackage{float}
\usepackage{pb-diagram}
\usepackage{color}

\newcommand{\br}[3]{{$#1$}$\lower4pt\hbox{$\tp\atop\raise4pt \hbox{$\scriptscriptstyle{#2}$}$} ${$#3$}}
\newcommand{\tw}[3]{{$#1$}${\,\scriptscriptstyle {#2}}\atop\raise9pt\hbox{$\scriptstyle\tp$} ${$#3$}}
\newcommand{\ttps}[2]{{#1}\raise5pt\hbox{$\lower12pt\hbox{$\scriptstyle\tp$}\atop \lower0pt\hbox{$\tilde\;$}$}\raise4.5pt\hbox{${\scriptstyle{#2}}$}}
\newcommand{\st}[1]{\mbox{${\,\scriptscriptstyle {#1}}\atop\raise5.5pt\hbox{$*$}$}}

\newcommand{\rd}[1]{\mbox{${\,\scriptscriptstyle {#1}}\atop\raise5.5pt\hbox{$\bullet$}$}}
\newcommand{\rt}[1]{\otimes_\chi}
\newcommand{\lt}[1]{\mbox{${\,\scriptscriptstyle {#1}}\atop\raise5.5pt\hbox{$\ltimes$}$}}
\newcommand{\btr}{\raise1.2pt\hbox{$\scriptstyle\blacktriangleright$}\hspace{2pt}}
\newcommand{\btl}{\raise1.2pt\hbox{$\scriptstyle\blacktriangleleft$}\hspace{2pt}}

\newcommand{\lcr}{\raise1.0pt \hbox{${\scriptstyle\rightharpoonup}$}}
\newcommand{\rcr}{\raise1.0pt \hbox{${\scriptstyle\leftharpoonup}$}}

\newcommand{\ttp}{{\lower12pt\hbox{$\tp$}\atop \hbox{$\tilde\;$}}}

\newcommand{\id}{\mathrm{id}}

\newcommand{\Arr}{{\mathrm{Arr}}}

\newcommand{\Cg}{\mathfrak{C}}

\newcommand{\gr}{ \mathrm{gr }}

\newcommand{\Ru}{\mathcal{R}}

\newcommand{\Q}{\mathcal{Q}}

\newcommand{\C}{\mathbb{C}}
\newcommand{\Z}{\mathbb{Z}}

\newcommand{\tp}{\otimes}
\newcommand{\vt}{\vartheta}

\newcommand{\ve}{\varepsilon}

\newcommand{\dt}{\delta}

\newcommand{\op}{\oplus}
\newcommand{\la}{\lambda}

\newcommand{\End}{\mathrm{End}}

\newcommand{\Tr}{\mathrm{Tr}}

\newcommand{\Rm}{\mathrm{R}}

\newcommand{\g}{\mathfrak{g}}
\renewcommand{\b}{\mathfrak{b}}
\renewcommand{\k}{\mathfrak{k}}
\newcommand{\h}{\mathfrak{h}}

\newcommand{\s}{\mathfrak{s}}
\renewcommand{\o}{\mathfrak{o}}

\newcommand{\nn}{\nonumber}
\newcommand{\p}{\mathfrak{p}}
\renewcommand{\l}{\mathfrak{l}}
\renewcommand{\c}{\mathfrak{c}}

\newcommand{\si}{\sigma}
\newcommand{\al}{\alpha}

\newcommand{\bt}{\beta}

\newcommand{\be}{\begin{eqnarray}}
\newcommand{\ee}{\end{eqnarray}}

\newtheorem{thm}{Theorem}[section]
\newtheorem{propn}[thm]{Proposition}
\newtheorem{lemma}[thm]{Lemma}
\newtheorem{corollary}[thm]{Corollary}

\newtheorem{definition}[thm]{Definition}

\newcount\prg

\newcommand{\parag}{\advance\prg by1 {\noindent\bf\thesection.\the\prg\hspace{6pt}}}

\begin{document}
\title{Representations of quantum conjugacy classes of orthosymplectic groups}
\author{
Thomas Ashton and Andrey Mudrov \vspace{20pt}\\
\small Department of Mathematics,\\ \small University of Leicester, \\
\small University Road,
LE1 7RH Leicester, UK\\
\small e-mail: am405@le.ac.uk\\
}

\date{}
\maketitle

\begin{abstract}
Let $G$ be the complex  symplectic or special orthogonal group and
 $\g$ its  Lie algebra.
With every  point $x$ of the maximal torus $T\subset G$ we associate a highest weight module $M_x$ over the Drinfeld-Jimbo quantum
group $U_q(\g)$ and a quantization of the conjugacy class of $x$ by operators in $\End(M_x)$.
These quantizations
are isomorphic for $x$ lying on the same orbit of the Weyl group, and $M_x$ support different representations of the same quantum conjugacy class.
\end{abstract}




\section{Introduction}
This paper is a sequel of a series of works on quantization of semisimple conjugacy classes of a non-exceptional simple Poisson
group $G$, \cite{DM}-\cite{AM1}.
It is done in the spirit of \cite{AM} devoted to  $G=SL(n)$ and can be viewed as a uniform approach to quantization
that includes the results of \cite{DM}-\cite{AM1} as a special case. The earlier constructed quantum conjugacy classes were realized by operators on certain modules of the quantized universal
enveloping algebra $U_q(\g)$ of the Lie algebra $\g$ of the group $G$.
For a large number of examples, this theory is parallel to the $U(\g)$-equivariant quantization of semisimple adjoint orbit in $\g\simeq \g^*$, \cite{DM,DGS,O}.
In both cases, $G$ and $\g$, the quantized algebra of  polynomial functions is represented on parabolic Verma modules, respectively, over
$U_q(\g)$ and $U(\g)$.
However, adjoint orbits in $G$ are in a greater supply than in $\g$.
Quantization of some of them requires more general modules, which cannot be obtained by induction from a character of the parabolic extension of
the stabilizer, \cite{M3,M4}. Moreover, the latter itself disappears as a natural subalgebra in $U_q(\g)$.
This observation makes us take a more general look at already constructed quantum homogeneous spaces and conclude that they
were obtained through a very special choice of the initial point. Such points are distinguished by their isotropy subgroups, whose triangular decomposition perfectly matches the triangular decomposition of $G$. All they are of Levi type, as for semisimple orbits in $\g$, and their basis of simple positive roots of is a part of
the basis of the total group.  That is violated for stabilizers of non-Levi type appearing among conjugacy classes in $G$. At the same time,
one can apply a generic Weyl group transformation to the initial point in $\g$ and break the nice inclusion of root bases in the Levi case. In this respect,
a generic initial point whose stabilizer is isomorphic to a Levi subgroup  has much similarity with essentially non-Levi one. It makes sense therefore to
extend the original approach to quantization and consider all points on the maximal torus (the Cartan subalgebra) for initial. They belong
to the same conjugacy class if and only if they lie on the same orbit of the Weyl group.
We associate a module of highest weight  with every such point
and realize the quantization of its cojugacy class  by linear operators on that module.
Points on the same Weyl group orbit give rise to isomorphic quantizations, which can be regarded as different representations of the same quantum homogeneous space.
They  can also  be thought of
as different polarizations of the same algebra.

There are other interesting problems related to quantum homogeneous spaces, such as quantization of  associated vector bundles,
star product formulation etc. That is well understood for classes with Levi isotropy subgroups,
through the mechanism of parabolic induction, \cite{DM1}-\cite{KST}. At the same time, the difference between Levi and non-Levi conjugacy classes is qualitative, and alternative representations of Levi classes could be a bridge between the two cases. A uniform approach to quantization may
help to understand the non-Levi case too.

\subsection{Preliminaries}
Let $G$ be the complex orthogonal or symplectic connected algebraic group of $N\times N$-matrices preserving a
 non-degenerate skew-diagonal symmetric or, respectively, symplectic form on $\C^N$.
Given a basis $\{w_i\}_{i=1}^N \in \C^N$, we fix the triangular decomposition $\g=\g_-\op \h\op \g_+$ so that the Cartan subalgebra is represented by diagonal matrices, while the nilpotent subalgebras $\g_\pm$ by strictly upper $(+)$ and lower $(-)$ triangular matrices. The basis elements $w_i$ carry weights $\ve_i\in \h^*$ satisfying  $\ve_{i'}=-\ve_i$, where $i'=N+1-i$.
Fix the inner product $(.,.)$ on $\h^*$ so that the weights with $i\leqslant \frac{N}{2}$ form an orthogonal basis.
Let $n$ designate the rank of $\g$. We choose  a basis $\Pi^+$ of simple roots in $\h^*$ as $\al_i=\ve_i-\ve_{i+1}$, $i<n$, and $\al_i=\ve_n$, $\al_i=2\ve_n$, $\al_i=\ve_{n-1}+\ve_n$ for, respectively,
$\g=\s\o(2n+1)$, $\g=\s\p(2n)$, and $\g=\s\o(2n)$. Denote by  $\Rm$ and $\Rm^+$ the sets of
all and positive roots of $\g$. When we need to distinguish the roots systems of a subgroup, we mark it with
the corresponding subscript.

Denote by $T$ the maximal torus of $G$ exponentiating the Cartan subalgebra $\h\subset \g$.
Given a point $x\in T$, denote by  $K\subset G$ its centralizer subgroup with the Lie algebra $\k$, which
 is a reductive subalgebra of maximal rank in $\g$.
The triangular decomposition of $\g$ induces a triangular decomposition $\k=\k_+\op \h\op \k_-$.
There are inclusions $\Rm_\k\subset \Rm_\g$ and $\Rm_\k^+\subset \Rm_\g^+$,
but not $\Pi^+_\k\subset \Pi^+_\g$ in general. If the latter holds,
$K$ is said to be a regular Levi subgroup of $G$.
If $K$ is not isomorphic to a Levi subgroup, we call it pseudo-Levi. We call it regular
if a maximal Levi subgroup among those contained in $K$ is regular.
Similar terminology is used for its Lie algebra $\k$.
Collectively we call $K$ and $\k$ generalized Levi subgroups and subalgebras.

The canonical inner product $(.,.)$ on the dual vector space $\h^*$ identifies it with $\h$. Let $h_\la\in \h$ denote
the image of $\la \in \h^*$ under this isomorphism.
Fix a generalized  Levi subalgebra $\k\subset \g$. By $\c_\k^*$ we denote the set of weights $\la\in \h^*$ such that
$(\la,\al)=0$ for all $\al \in \Rm_\k$ and by
$\c^*_{\k,reg} \subset \c^*_\k$ the set of weights such that $(\la,\al)=0 \Leftrightarrow \al \in \Rm_\k$.
For each $\la \in \c^*_\k$ the element $e^{2h_\la}\in G$ commutes with $K$,  and $\k$ is exactly the centralizer
Lie algebra
of $x=e^{2h_\la}$ once $\la \in \c^*_{\k,reg}$.

The coordinate ring $\C[O_x]$ is a quotient of  $\C[G]$ by a certain $G$-invariant ideal.
To describe this ideal, observe that $x$ determines a 1-dimensional representation $\chi_x$ of the subalgebra of invariants in
$\C[G]$ (under the conjugation action).
Apart from $SO(2n)$, it is generated by traces of the matrix powers of $(X_{ij})$, where $X_{ij}$ are the coordinate functions
on $G$. In the special case of  $SO(2n)$ one has to add one more invariant that is sensible to the flip of the Dynkin diagram, in order to separate two $SO(2n)$-classes within a $O(2n)$-class whose eigenvalues are all distinct from $\pm 1$. Furthermore, the
matrix $X$, when restricted to $O_x$, satisfies an equation  $p(X)=0$ with a polynomial $p$ in one variable. The entries of the matrix $p(X)$ are polynomial functions in $X_{ij}$. The defining ideal of $O_x$ is generated by the entries of $p(X)$  over the kernel of $\chi_x$, provided $p$
is the minimal polynomial for $x$.

A pseudo-Levi subgroup $K$ contains a Cartesian product of two blocks of the same type as $G$.
They correspond to the eigenvalues $\pm 1$ of the matrix $x$, which are simultaneously present in its spectrum.
For the symplectic group,  it is $SP(2m)\times SP(2p)$, where $m, p\geqslant 1$.
For the odd orthogonal group, it is $SO(2m)\times SO(2p+1)$, where $m\geqslant 2$, $p\geqslant 0$.
For the even orthogonal group, one has $SO(2m)\times SO(2p)$, where $m,p\geqslant 2$.
The lower bounds on $m,p$ come from the isomorphism $SO(2)\simeq GL(1)$: if the multiplicities of $\pm 1$ are small, then
the isotropy subgroup stays within the Levi type. We distinguished such conjugacy classes as borderline Levi because they share some properties of both types, \cite{AM1}.

The quantized polynomial algebra $\C_\hbar[O_x]$,  $\hbar=\log q$, is described as follows. The algebra $\C[G]$ is replaced with $\C_\hbar[G]$,
which is an equivariant quantization of a special Poisson bracket on $G$, \cite{STS}. This bracket makes $G$ a Poisson-Lie homogeneous space
over the Poisson group $G$ equipped with the Drinfeld-Sklyanin bracket \cite{D}, with respect to the conjugation action. The algebra $\C_\hbar[G]$ admits
an equivariant embedding into the corresponding quantum group $U_\hbar(\g)\supset U_q(\g)$. As a subalgebra in $U_\hbar(\g)$, it is generated
by the entries of the matrix $\Q=(\pi\tp \id)(\Ru_{21}\Ru)$, where $\Ru$ is the universal R-matrix of $U_\hbar(\g)$
and $\pi$ stands for the representation homomorphism $U_q(\g)\to \End(\C^N)$.
The factor $\Ru_{21}$ is obtained by flip of the tensor legs of $\Ru$. This embedding makes a $U_q(\g)$-module into a  $\C_\hbar[G]$-module and
the representation homomorphism of $\C_\hbar[G]$ automatically $U_q(\g)$-equivariant.

The subalgebra of invariants in $\C_\hbar[G]$ coincides with its centre, which
is generated by q-traces of the matrix powers of $\Q$ (apart from the special case of $SO(2n)$, as mentioned above).
The ``quantum initial points'' can be described as follows.
Let $\rho_\k=\frac{1}{2}\sum_{\al\in \Ru_\k^+}\al$ be the Weyl vector of the isotropy subalgebra $\k$.
Let $\c_\k^*$ be the orthogonal complement to $\C\Pi^+_\k$.
Denote $\Cg^*_{\k,reg}=\frac{1}{\hbar}\c^*_{\k,reg}+\c^*_{\k}+\rho_\k-\rho$ and
 $\Cg^*_{\k}=\frac{1}{\hbar}\c^*_{\k}+\c^*_{\k}+\rho_\k-\rho$.
By construction,  $\Cg^*_{\k}$ is the subset of $\la\in \frac{1}{\hbar}\h^*\op \h^*$ such that $q^{2(\la+\rho,\al)}=q^{(\al,\al)}$  for all $q$  if  $\al \in \Pi^+_\k$ while $\Cg^*_{\k,reg}\subset \Cg^*_{\k}$ satisfies this condition only if  $\al \in \Pi^+_\k$.

With $\la \in \Cg^*_{\k,reg}$  we associate a module $M_\la$ of highest weight $\la$,
so that the image of $\C_\hbar[G]$ in $\End(M_\la)$ is a quantization of $\C_\hbar[O_x]$.
It is a parabolic Verma module if and only if $\k$ is a regular Levi subalgebra. Irregular Levi subgroups
also appear as stabilizers of initial points in $\g$, so our approach is as well applicable to the
$U(\g)$-equivariant quantization of adjoint orbits in $\g$.

The highest weight of $M_\la$
defines a central character
of $\C_\hbar[G]$, whose kernel is expressed through q-traces of the matrix powers $\Q^k$.
The matrix $\Q$ yields an invariant operator on  $\C^N\tp M_\la$, and its minimal polynomial  is determined by module structure of the tensor product.
The annihilator of $M_\la$ is then generated by the entries of the minimal polynomial over the kernel of the central character.
The structure of  $\C^N\tp M_\la$ is the key  point of this approach, and  its analysis takes a great part of this exposition. This study is makes use of some results on the Mickelsson algebras and Shapovalov inverse \cite{AM2,M7} and based on  the study of the standard filtration of $\C^N\tp M_\la$ in what follows.

\subsection{Quantized universal enveloping algebra}
\label{ssecQUEA}
Throughout the paper, $\g$ is  a complex simple Lie algebra of type $B$, $C$ or $D$ (the $A$-case has been considered in \cite{AM}).
We assume that $q\in \C$ is not a root of unity. Denote by $U_q(\g_\pm)$ the $\C$-algebra generated by  $e_{\pm\al}$, $\al\in \Pi^+$, subject to the q-Serre relations, \cite{CP}.
Denote by $U_q(\h)$ the commutative $\C$-algebra generated by $q^{\pm h_\al}$, $\al\in \Pi^+$. The quantum group $U_q(\g)$ is a $\C$-algebra generated by  $U_q(\g_\pm)$ and $U_q(\h)$ subject
to the relations
$$
q^{ h_\al}e_{\pm \bt}q^{-h_\al}=q^{\pm(\al,\bt)} e_{\pm \bt},
\quad
[e_{\al},e_{-\bt}]=\delta_{\al, \bt}\frac{[h_{\al}]_q}{[\frac{(\al,\al)}{2}]_q},
$$
were  $[z]_q=\frac{q^z-q^{-z}}{q-q^{-1}}$. We work with the opposite comultiplication  as in  \cite{CP}:
\be
&\Delta(e_{\al})=e_{\al}\tp 1 + q^{h_{\al}}\tp e_{\al},
\quad
\Delta(e_{-\al})=e_{-\al}\tp q^{-h_{\al}} + 1 \tp e_{-\al},
\nn\\&
\Delta(q^{\pm h_{\al}})=q^{\pm h_{\al}}\tp q^{\pm h_{\al}},
\nn
\ee
for all $\al \in \Pi^+$.
The quantized Borel subalgebras $U_q(\b_\pm)\subset U_q(\g)$,  $\b_\pm=\g_\pm+\h$,  are generated by $U_q(\g_\pm)$ over $U_q(\h)$.
The universal R-matrix is fixed to be an element of an extended   tensor product of $U_q(\b_-)\tp U_q(\b_+)$.
Its transposed version due to the opposite comultiplication can be taken from \cite{CP}, Theorem 8.3.9.

We use the notation $e_i=e_{\al_i}$, and $f_{i}=e_{-\al_i}$ for  $\al_i\in \Pi^+$ in all cases apart
from $i=n$, $\g=\s\o(2n+1)$, where we set $f_n=[\frac{1}{2}]_q e_{-\al_n}$.
The corresponding commutation relation translates to
$
[e_{n},f_{n}]=[h_{\al_n}]_q.
$
With this normalization of generators, the  natural representation of $U_q(\g)$ on the vector space $\C^N$
is independent of $q$, see the next section.

\section{Natural representation of $U_q(\g)$}
By $\Gamma$ we denote the root lattice $\Gamma=\Z\Pi^+$ with  $\Gamma^+=\Z_+\Pi^+$.
Let $I$ designate the set of integers $\{1,\ldots,N\}$. For $\bt\in \Gamma^+$ we define $P(\bt)$ to be the set of all pairs $i,j\in I$
such that $\ve_i-\ve_j=\bt$.
Let $e_{ij}\in  \End(\C^N)$, $i,j\in I$, denote the standard matrix units. The following assignment defines a representation of $\g$, which
is equivalent to the natural representation:
$\pi(e_\al)=\sum_{(l,r)\in P(\al)}e_{lr}$,
$\pi(f_\al)=\sum_{(l,r)\in P(\al)}e_{rl}$, $\pi(h_{\ve_i})=e_{ii}-e_{i'i'}$.
The action of the Chevalley generators can be conveniently visualized by the diagrams
\begin{center}
\begin{picture}(390,60)
\put(160,45){$\g=\s\o(2n+1)$}
\put(0,0){$w_{1'}$}
\put(10,15){\circle{3}}
\put(45,15){\vector(-1,0){30}}
\put(50,15){\circle{3}}
\put(100,15){$\ldots$}
\put(85,15){\vector(-1,0){30}}
\put(45,0){$w_{2'}$}
\put(160,15){\circle{3}}
\put(155,15){\vector(-1,0){30}}
\put(150,0){$w_{n'}$}
\put(195,15){\vector(-1,0){30}}
\put(200,15){\circle{3}}
\put(190,0){$w_{n+1}$}
\put(235,15){\vector(-1,0){30}}
\put(240,15){\circle{3}}
\put(230,0){$w_{n}$}
\put(275,15){\vector(-1,0){30}}
\put(345,15){\vector(-1,0){30}}
\put(350,15){\circle{3}}
\put(345,0){$w_{2}$}
\put(290,15){$\ldots$}
\put(385,15){\vector(-1,0){30}}
\put(390,15){\circle{3}}
\put(385,0){$w_{1}$}
\put(25,20){$f_{\al_1}$}
\put(65,20){$f_{\al_2}$}
\put(135,20){$f_{\al_{n-1}}$}
\put(175,20){$f_{\al_{n}}$}
\put(215,20){$f_{\al_{n}}$}
\put(255,20){$f_{\al_{n-1}}$}
\put(325,20){$f_{\al_{2}}$}
\put(365,20){$f_{\al_{1}}$}

\end{picture}
\end{center}
\begin{center}
\begin{picture}(350,60)
\put(155,45){$\g=\s\p(2n)$}
\put(0,0){$w_{1'}$}
\put(10,15){\circle{3}}
\put(45,15){\vector(-1,0){30}}
\put(50,15){\circle{3}}
\put(100,15){$\ldots$}
\put(85,15){\vector(-1,0){30}}
\put(45,0){$w_{2'}$}
\put(160,15){\circle{3}}
\put(155,15){\vector(-1,0){30}}
\put(150,0){$w_{n'}$}
\put(195,15){\vector(-1,0){30}}
\put(200,15){\circle{3}}
\put(190,0){$w_{n}$}
\put(235,15){\vector(-1,0){30}}
\put(305,15){\vector(-1,0){30}}
\put(310,15){\circle{3}}
\put(305,0){$w_{2}$}
\put(250,15){$\ldots$}
\put(345,15){\vector(-1,0){30}}
\put(350,15){\circle{3}}
\put(345,0){$w_{1}$}
\put(25,20){$f_{\al_1}$}
\put(65,20){$f_{\al_2}$}
\put(135,20){$f_{\al_{n-1}}$}
\put(175,20){$f_{\al_{n}}$}
\put(215,20){$f_{\al_{n-1}}$}
\put(285,20){$f_{\al_{2}}$}
\put(325,20){$f_{\al_{1}}$}

\end{picture}
\end{center}
\begin{center}
\begin{picture}(440,70)
\put(195,55){$\g=\s\o(2n)$}
\put(0,0){$w_{1'}$}
\put(10,15){\circle{3}}
\put(50,15){\circle{3}}
\put(45,15){\vector(-1,0){30}}
\put(85,15){\vector(-1,0){30}}
\put(100,15){$\ldots$}
\put(155,15){\vector(-1,0){30}}
\put(45,0){$w_{2'1}$}
\put(160,15){\circle{3}}
\put(150,0){$w_{n'+1}$}
\put(195,15){\vector(-1,0){30}}
\put(200,15){\circle{3}}
\put(190,0){$w_{n'}$}
\qbezier(163,22)(200,48)(237,22) \put(162,21.5){\vector(-2,-1){2}}
\qbezier(203,22)(240,48)(277,22) \put(202,21.5){\vector(-2,-1){2}}
\put(240,15){\circle{3}}
\put(235,0){$w_{n}$}
\put(275,15){\vector(-1,0){30}}
\put(280,15){\circle{3}}
\put(270,0){$w_{n-1}$}
\put(330,15){$\ldots$}
\put(315,15){\vector(-1,0){30}}
\put(385,15){\vector(-1,0){30}}
\put(390,15){\circle{3}}
\put(385,0){$w_{2}$}
\put(425,15){\vector(-1,0){30}}
\put(430,15){\circle{3}}
\put(425,0){$w_{1}$}
\put(25,20){$f_{\al_1}$}
\put(65,20){$f_{\al_2}$}
\put(135,20){$f_{\al_{n-2}}$}
\put(172,20){$f_{\al_{n-1}}$}
\put(195,38){$f_{\al_{n}}$}
\put(252,20){$f_{\al_{n-1}}$}
\put(290,20){$f_{\al_{n-2}}$}
\put(235,38){$f_{\al_{n}}$}
\put(365,20){$f_{\al_{2}}$}
\put(405,20){$f_{\al_{1}}$}
\end{picture}
\end{center}
Reverting the arrows one gets the diagrams for positive Chevalley generators of $\g$.

We introduce a partial ordering on the integer interval $I$ by setting $i\preccurlyeq j$ if and only if
there is a (monic) Chevalley monomial $\psi \in U_q(\g_-)$ such that  $w_j$ is equal to $\psi w_i$ up to an invertible scalar multiplier,
$w_j=\psi w_i$. This monomial, if exists, represents a path from $w_i$ to $w_j$ in the representation diagram, which
becomes the Hasse diagram of the poset.
Such $\psi$ is unique, which is obvious for the series $B$ and $C$ and still true for $D$. Indeed,
two different paths from $w_{n-1}$ to $w_{n+2}$ yield the products $f_{\al_n}f_{\al_{n-1}}$ and $f_{\al_{n-1}}f_{\al_n}$,
which are the same due to Serre relations. We denote this monomial by $\psi_{ji}$.
The relation $\prec$ is consistent with the natural ordering on $\Z$, and coincides with it unless
 $\g=\s\o(2n)$. In the latter case $n$ and $n'$ are incomparable.


In what follows, we also use the monomials  $\psi^{ij}$ obtained from  $\psi_{ji}$ by reverting the order of factors.
It is clear that $\psi^{ij}=\psi^{im}\psi^{mj}$ for any $m$ such that $i\preccurlyeq m\preccurlyeq j$.
\begin{definition}
We call $\psi^{ij}$  the principal  monomial of the pair $i\preccurlyeq j$.
\end{definition}
\noindent
Their significance will be clear later in the
section devoted to the standard filtration of tensor product modules.

Remark that all Chevalley monomials of weight $\ve_j-\ve_i$ are obtained from $\psi^{ij}$ by permutation of factors.

We will also need another partial ordering on $I$ that is relative to $\k$: write
$i\lessdot j$ if  $w_i$ and $w_j\in U_q(\k_-)\k_-w_i$. Clearly $i\lessdot j$ if and only if
$i\prec j$ and $w_i,w_j$ belong an irreducible $\k$-submodule in $\C^N$.
Let $I_\k\subset I$  be the set of all minimal elements with respect to this ordering and
$\bar I_\k$ be its complement in $I$. Elements of $I_\k$ label the highest weight vectors
of the irreducible $\k$-submodules in $\C^N$.

\subsection{Reduced Shapovalov inverse}
\label{Sec:RedShapInf}
In this section, we recall a construction of  Shapovalov inverse reduced to $\End(\C^N)\tp U_q(\b_-)$.
It is given in \cite{AM2} for the general linear and orthosymplectic quantum groups (see also \cite{M7} for the general case). Note with care that \cite{AM2,M7} deal with a different version of the quantum group.
To adapt those results to the current setting, one has to twist the coproduct by $q^{\sum_{i=1}^n h_i\tp h_i}$ and
replace $q$ with $q^{-1}$.

Given $\la \in \frac{1}{\hbar}\h^*\op\h^*$ consider a 1-dimensional $U_q(\b_\pm)$-module $\C_\la$ with the representation  defined by the assignment
$q^{\pm h_\al}\mapsto q^{(\la,\al)}$, $e_\al\mapsto 0$ for $\al\in \Pi^+$.
Denote by  $M_\la$ the Verma module $U_q(\g)\tp_{U_q(\b_+)}\C_\la$ with the canonical generator $v_\la$, \cite{Ja}.
Let $M_\la^*$ denote the opposite Verma module $U_q(\g)\tp_{U_q(\b_-)}\C_{-\la}$ of the lowest weight $-\la$.
There is an invariant pairing $M_\la\tp M_\la^*\to \C$, which is equivalent to the contravariant Shapovalov form on $M_\la$,
upon an identification $M_\la^*\sim M_\la$ through an anti-algebra isomorphism $U_q(\g_-)\simeq U_q(\g_+)$, \cite{DCK}. We also call it Shapovalov form.

Recall that a vector $v\not =0$ in a $U_q(\g)$-module $V$ is called singular if $e_\al v=0$ for all $\al \in \Pi^+$.
Singular vectors are defined up to a scalar multiplier.
Reduced Shapovalov inverse is a matrix $\hat F=\sum_{i=1}^j e_{ij}\tp \hat f_{ij} \in \End(\C^N)\tp \hat U_q(\b_-)$,
where the roof means extension over the field of fractions of $U_q(\h)$. This matrix yields
a singular vector $\hat F(w_j\tp v_\la)$ in $\C^N\tp M_\la$ for all $j\in I$.
For generic $\la$ the matrix $\hat F$ is a homomorphic image
of the  Shapovalov inverse  lifted to $\hat U_q(\g_+)\tp \hat U_q(\b_-)$.

The entries $\hat f_{ij}$ can be expressed through the Chevalley generators as follows.
First introduce $f_{ij}\in U_q(\g_-)$ for all $i<j$, which are closely related
to the R-matrix of $U_q(\g)$, \cite{AM2}. Put $f_{ij}=f_{j'i'}=f_i$ for $i-1=j<\frac{N+1}{2}$ and
\be
f_{ij}=[f_{j-1},\ldots [f_{i+1},f_i]_{q^{}}\ldots ]_{q},
\quad
f_{j'i'}=[\ldots [f_i,f_{i+1}]_{q^{}},\ldots f_{j-1}]_{q},
\label{gl(n)}
\ee
for  $i+1<j\leqslant\frac{N+1}{2}$ and all $\g$. Furthermore,
$$
f_{nn'}=(q^{-1}-1)f_n^2, \quad f_{i,n+1}=[f_{n},f_{in}]_{q}, \quad f_{n+1,i'}=[f_{n'i'},f_{n}]_{q}, \quad i<n
,\quad \g=\s\o(2n+1),
$$
$$
f_{nn'}=[2]_qf_n, \quad f_{in'}=[f_{n},f_{in}]_{q^{2}}, \quad f_{ni'}=[f_{n'i'},f_{n}]_{q^{2}}, \quad i<n,
\quad
\g=\s\p(2n),
$$
$$
f_{nn'}=0, \quad f_{in'}=[f_{n},f_{i,n-1}]_{q}, \quad f_{ni'}=[f_{n'+1,i'},f_n]_{q},\quad i<n-2
,\quad \g=\s\o(2n),
$$
and finally, for $ i,j<n$,
$$
f_{ij'}=q^{-\dt_{ij}}[f_{n+1,j'},f_{i,n+1}]_{q^{\dt_{ij}}},  \> N=2n+1
,\qquad
f_{ij'}=q^{-\dt_{ij}}[f_{nj'},f_{in}]_{q^{1+\dt_{ij}}},  \> N=2n.
$$
There exists an analog of Poincare-Birghoff-Witt (PBW) basis in $U_q(\g_-)$ generated by certain
elements labeled by $\Rm^+$, which can be presented as deformed commutators of the Chevalley generators, \cite{CP}.
The presence of PBW bases allows to identify $U_q(\g_-)$ with $U(\g_-)$ as vector spaces (and $U_q(\h)$-modules).
This identification makes $U_q(\g_-)$ a deformation of $U(\g_-)$. It follows that $f_{ij}$ are deformations of root vectors from  $\g_-$.

Put $\rho_i=(\rho,\ve_i)$ for $i\in I$ and introduce $\eta_{ij}=h_i-h_j+\rho_i-\rho_j-\frac{||\ve_i-\ve_j||^2}{2}\in \h+\C$,
$
A^j_i=-\frac{q-q^{-1}}{q^{2\eta_{ij}}-1},
$
for all $i,j\in I$ such that $i\prec j$.
We call a sequence $\vec m=(m_1,\ldots,m_k)$ a route from $m_1$ to $m_k$ if $m_1\prec \ldots \prec m_k$.
To every route $\vec m$ we assign
the products
$$
f_{\vec m}=f_{m_1,m_2}\ldots f_{m_{k-1},m_k}, \quad A^j_{\vec m}=A_{m_1}^j\ldots A_{m_k}^j,
$$
where $m_k\prec j$.
Given another route, $\vec l=(l_1,\ldots,l_s)$ with $\vec m\prec \vec l$ meaning $m_k\prec l_1$,
there is a route $(\vec m, \vec l)=(m_1,\ldots,m_k, l_1,\ldots,l_s)$.
 Define $\tilde \rho_i=\rho_i+\frac{||\ve_i||^2}{2}$ for all $i\in I$.
Then $\hat f_{ij}=0$ if $i> j$,
$\hat f_{ii}=1$ and
$
\hat f_{ij}=\sum_{i \preccurlyeq  \vec m \prec j} f_{\vec m,j}A_{\vec m}^j
q^{\eta_{ij}-\tilde \rho_i+\tilde \rho_j}
$ for $i<j$,
where the summation is done over all routes $(\vec m, j)$ from
$i$ to $j$. Note that the factor $q^{\eta_{ij}-\tilde \rho_i+\tilde \rho_j}$ comes from a different version of
the quantum group adopted in \cite{AM,M7}.
\begin{lemma}
\label{dynamical_roots_deformed}
Suppose that $\al \in \Pi^+_\k\subset \Rm_\g^+$ and $(i,j)\in P(\al)$.
For all $\la\in \Cg^*_{\k,reg}$, the specialization $\hat{f}_{ij}[\eta_{ij}]_q$ at weight $\la$ is a deformation of
a classical root vector, $-f_{\al}\in \g_-$.
\end{lemma}
\begin{proof}
Present $\la$ as $\al=\frac{1}{\hbar} \la^0+\la^1\in \Cg^*_{\k,reg}$, $\la^i\in \h^*$.
Observe that a) $e^{2\la^0_i}=e^{2\la^0_j}$ for all $\al=\ve_i-\ve_j\in \Pi_\k^+$ once
$\la^0\in \c^*_{\k}$ and b) there is no $k$ such that
$i\prec k \prec j$ and $e^{2\la^0_i}=e^{2\la^0_k}=e^{2\la^0_j}$ if $\la^0\in \c^*_{\k,reg}$.
Furthermore, write
$\hat{f}_{ij}[\eta_{ij}]_q=-f_{ij}- \sum_{i\prec \vec m\vec \prec j } f_{i,\vec m,j}A^{j}_{\vec m,j}q^{\tilde \rho_j-\tilde \rho_i}$, where the sum is taken over
non-empty routes $\vec m$.
For all $k$ subject to  $i\prec  k\prec j$, the denominator in $A^j_k|_{\la}=-\frac{q-q^{-1}}{q^{2\eta_{kj}|_\la}-1}$ tends to $e^{2\la^0_k-2\la^0_j}-1\not =0$ as $q\to 1$. Therefore, the sum vanishes modulo $\hbar$,
and $f_{ij}$ tends to a classical root vector.
\label{non-zero}
\end{proof}
Define elements $\check{f}_{ij}=\hat f_{ij}\prod_{i\preccurlyeq k \prec j}[\eta_{kj}]_q\in U_q(\b_-)$ for all $i\prec j$.
They satisfy the identity
\be
e_\al \check{f}_{ij}=-\sum_{(l,r)\in P(\al)}\dt_{l,i}q^{-(\al,\ve_l)}\check{f}_{r,j}[\eta_{ij}]_q \mod U_q(\g)\g_-, \quad \forall\al \in \Pi^+,
\label{e_on_f}
\ee
Fix $(i,j)\in P(\al)$ for  $\al \in \Rm^+$ and suppose that $\la=\frac{1}{\hbar}\la^0+\la^1$ with  $\la^i\in \h^*$ satisfies the condition $[\eta_{ij}|_\la]_q=0= [\eta_{j'i'}|_\la]_q$.
Then there is a singular vector $v_{\la-\al}$ of weight $\la-\al$ in the Verma module $M_\la$.
One can take $v_{\la-\al}=\check{f}_{ij}v_\la$ provided it is not zero,  since
$e_\al \check{f}_{ij}v_\la=0$ for all $\al \in \Pi^+$ by (\ref{e_on_f}).
If $\check{f}_{ij}v_\la=0$  at some $\la$,  one still can obtain $v_{\la-\al}$ from $\check{f}_{ij}v_\la$
(which is polynomial in $e^{\pm 2(\la^0,\al)}$, $\al \in \Pi^+$, for fixed $\la^1$ and $q$)
via renormalization,  since
singular vectors are defined up to a scalar multiplier.
In particular, if  $\al \in \k$ for some generalized Levi subalgebra $\k$ and $\la \in \Cg^*_{\k,reg}$,
then $v_{\la-\al}\simeq f_\al v_{\la}\mod \hbar $, by Lemma \ref{dynamical_roots_deformed}.
Note that $\check{f}_{ij}v_\la\simeq \check{f}_{j'i'}v_\la$ if $i\not = j'$,
as follows from the theory of Mickelsson algebras for quantum groups, \cite{KO}.
\section{Standard filtration on $\C^N\tp M_\la$}
In what follows, we work out a tool for our analysis of $\C^N\tp M_\la$, where
$M_\la$ is a generalized parabolic Verma module of weight $\la$. In this section, we do it for the ordinary Verma module
 $M_\la = U_q(\g)\tp_{U_q(\b_+)}\C_\la$ with $\la \in \frac{1}{\hbar}\h^*\op\h^*$.
An essential part of our technique  is a diagram language,  whose elements already appeared in \cite{M1,M4} and which is
given a systematic treatment here. The case of $\g\l(N)$ was already studied in \cite{AM}, so
we do it for orthogonal and symplectic $\g$.
We consider the standard filtration $V_\bullet=(V_i)_{i=1}^N$, $\{0\}=V_0\subset V_1\subset \ldots \subset V_N= \C^N\tp M_\la$,
where $V_i$ is generated by $\{e_j\tp v_\la\}$, $j\leqslant i$. Its graded module $\gr V_\bullet$ is a direct
sum of $V_j/V_{j-1}$, which are isomorphic to the Verma modules $M_{\la+\ve_j}$ (the  proof of \cite{BGG}, Lemma 5, readily adapts to quantum groups).

Given $\bt\in \Z_+\Pi^+$ we define  $\Psi_{\bt}\subset U_q(\g_-)$ to be the subset of Chevalley monomials of weight
$\bt$. We assume that a pair $(i,j)\in P(\bt)$ is chosen for this section.
Having fixed an order of elementary factors in $\psi$, we regard it a as path from $v_\la$ to $\psi v_\la$.
 We associate with $\psi v_\la$ a graph $H_{\psi}$ with nodes $\{v^k\}\in M_\la$,
$v^j=v_\la$, $v^i=\psi v_\la$, and arrows being negative Chevalley generators acting on $M_\la$.
For $\psi=\psi^{ij}$, this path is unique in almost all cases (except for type $D$, where we eliminate the ambiguity by fixing the order as $f_{\al_{n-1}}f_{\al_n}$). For principal $\psi$, we are concerned not just with  the terminating node $\psi v_\la$,
but also in all intermediate nodes. On the contrary, for non-principal $\psi$, only  $\psi v_\la$ is important for us, while the
specific path is immaterial.

We say that $f_\al$ has length $2$ if $\al=\al_n$ and $\g=\s\o(2n)$. All other generators are assigned with length $1$.
If all factors in $\psi$ have length $1$, we write $\psi= \phi_i\ldots \phi_{j-1}$ with $\phi_k\in \{f_\al\}_{\al\in \Pi^+}$, and we put
set $v^k=\phi_k v^{k+1}$. Then the diagram $H_{\psi}$ is set to be
\begin{center}
\begin{picture}(200,30)
\put(7,0){$v^{i}$}
\put(10,15){\circle{3}}
\put(45,15){\vector(-1,0){30}}
\put(50,15){\circle{3}}
\put(100,15){$\ldots$}
\put(85,15){\vector(-1,0){30}}
\put(45,0){$v^{i+1}$}
\put(160,15){\circle{3}}
\put(155,15){\vector(-1,0){30}}
\put(152,0){$v^{j-1}$}
\put(195,15){\vector(-1,0){30}}
\put(200,15){\circle{3}}
\put(197,0){$v^{j}$}
\put(25,20){$\phi_{i}$}
\put(65,20){$\phi_{\al_{i+1}}$}
\put(135,20){$\phi_{j-2}$}
\put(175,20){$\phi_{j-1}$}

\end{picture}
\end{center}
Now suppose that $\psi$ has (exactly one)  factor of length $2$. Write  $\psi=\phi_i\ldots \phi_{k}\phi_{k+2}\ldots \phi_{j-1}$,
where $\phi_{k}=f_{\al_n}$ (there are $j-i-1$ factors). Then the graph $H_{\psi}$  is
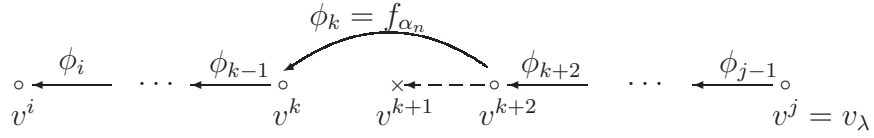
\begin{center}
\begin{picture}(290,40)
\put(7,0){$v^{i}$}
\put(10,15){\circle{3}}
\put(45,15){\vector(-1,0){30}}
\put(55,15){$\ldots$}
\put(105,15){\vector(-1,0){30}}
\put(110,15){\circle{3}}
\put(106,0){$v^{k}$}
\qbezier(113,22)(150,48)(187,22) \put(112,21.5){\vector(-2,-1){2}}
\put(150,13){$\scriptstyle{\times}$}
\put(162,15){\vector(-1,0){6}}\put(170,15){\line(-1,0){5}}\put(178,15){\line(-1,0){5}}\put(186,15){\line(-1,0){5}}
\put(145,0){$v^{k+1}$}
\put(190,15){\circle{3}}
\put(184,0){$v^{k+2}$}
\put(240,15){$\ldots$}
\put(225,15){\vector(-1,0){30}}
\put(295,15){\vector(-1,0){30}}
\put(300,15){\circle{3}}
\put(295,0){$v^{j}=v_\la$}
\put(25,21){$\phi_{i}$}
\put(82,20){$\phi_{k-1}$}
\put(200,20){$\phi_{k+2}$}
\put(120,38){$\phi_{k}=f_{\al_{n}}$}
\put(275,20){$\phi_{j-1}$}
\end{picture}
\end{center}
Here we distinguish two cases. If  $\psi=\psi^{ij}$, then $\phi_{k-1}=f_{\al_{n-1}}$, and the dashed arrow $f_{\al_{n-1}}$ is included
in $H_\psi$.
The node $v^{k+1}$ is set to $f_{\al_{n-1}}v^{k+2}$. For non-principal $\psi$, the node $v^{k+1}$ is arbitrary (immaterial) and there is no arrow
from $v^{k+2}$ to $v^{k+1}$.

We also consider a graph $V_{ij}$, which is a  part of
the natural representation diagram of $U_q(\g_-)$ that includes all paths from $w_i$ to $w_j$. We transpose it to make a vertical graph oriented from
top $w_i$ to bottom $w_j$.

We denote by $\Arr(v^k)$ the set of arrows originated at $v^k$ and similarly $\Arr(w_m)$ the set of arrows from $w_k$.
By construction, an arrow from node $m$ to node $k$ has length $k-m$.

Finally, we define tensor product $D_\psi=H_\psi\tp V_{ij}$ as a  graph on a two-dimensional lattice
whose nodes are $w_k^m=w_k\tp v^m \in \C^N\tp  M_\la$ and arrows are $\Arr(w_k^m)= \Arr (w_k)\tp \id\bigcup \id\tp \Arr (v^m)$,
The diagram is oriented so that $H_\psi$-arrows and $V_{ij}$-arrows  are directed, respectively, leftward and downward; the origin $w^j_i$ is in the right upper corner.
 We need only
the triangular part of the diagram including the nodes $v^k_m$ with $k+m\geqslant i+j$. The set
$\{w_k^k\}_{k=i}^j$ is called principal diagonal. With $\psi=\psi^{ij}$, the node $w^k_k$ on
the principal diagonal is $w_k\tp \psi^{kj}v_\la$, $k=i,\ldots, j$.
Here is an example of diagram $D_\psi$ with all arrows of length $1$:
\begin{figure}[H]
\caption{}
\label{D_psi}
\begin{center}
\begin{picture}(290,180)
\put(150,20){}
\put(5,160){$\scriptstyle{w^{i}_{i}}$}
\put(60,160){$\scriptstyle{w^{i+1}_{i}}$}
\put(115,160){$\scriptstyle{w^{i+2}_{i}}$}
\put(225,160){$\scriptstyle{w^{j-1}_i}$}
\put(275,160){$\scriptstyle{w^j_i}$}

\put(55,160 ){\vector(-1,0){30}}\put(108,160 ){\vector(-1,0){30}}\put(163,160 ){\vector(-1,0){30}}
\put(170,160 ){$\ldots$}
\put(218,160 ){\vector(-1,0){30}}
\put(270,161 ){\vector(-1,0){30}}

\put(40,167){$\scriptstyle{\phi_{i} }$}\put(91,167){$\scriptstyle{\phi_{i+1}}$}
\put(148, 167){$\scriptstyle{\phi_{i+2}}$}\put(201,167){$\scriptstyle{\phi_{j-2}}$}\put(250,167){$\scriptstyle{\phi_{j-1}}$}

\put(63,157 ){\vector(0,-1){18}}\put(119,157){\vector(0,-1){18}}\put(228,157 ){\vector(0,-1){18}}
\put(278,157 ){\vector(0,-1){18}}

\put(60,130){$\scriptstyle{w^{i+1}_{i+1}}$}
\put(115,130){$\scriptstyle{w^{i+2}_{i+1}}$}
\put(225,130){$\scriptstyle{w^{j-1}_{i+1}}$}
\put(275,130){$\scriptstyle{w^j_{i+1}}$}

\put(108,131 ){\vector(-1,0){30}}\put(163,131 ){\vector(-1,0){30}}
\put(170,131 ){$\ldots$}
\put(218,131 ){\vector(-1,0){30}}
\put(270,131 ){\vector(-1,0){30}}

\put(119,127){\vector(0,-1){18}}\put(228,127 ){\vector(0,-1){18}}
\put(278,127 ){\vector(0,-1){18}}

\put(115,100){$\scriptstyle{w^{i+2}_{i+2}}$}
\put(225,100){$\scriptstyle{w^{j-1}_{i+2}}$}
\put(275,100){$\scriptstyle{w^j_{i+2}}$}

\put(163,101 ){\vector(-1,0){30}}
\put(170,101 ){$\ldots$}
\put(218,101 ){\vector(-1,0){30}}
\put(270,101 ){\vector(-1,0){30}}

\put(228,97 ){\vector(0,-1){18}}
\put(278,97 ){\vector(0,-1){18}}
\put(227,65 ){$\vdots$}
\put(277,65 ){$\vdots$}

\put(187,67){\circle{1}}\put(177,72){\circle{1}}\put(167,77){\circle{1}}

\put(228,60 ){\vector(0,-1){18}}\put(278,60 ){\vector(0,-1){18}}
\put(278,32 ){\vector(0,-1){18}}

\put(225,35){$\scriptstyle{w^{j-1}_{j-1}}$}\put(274,35){$\scriptstyle{w_{j-1}^j}$}
\put(270,38 ){\vector(-1,0){30}}

\put(274,5){$\scriptstyle{w_{j}^j}$}


\end{picture}
\end{center}
\end{figure}

\noindent
The arrows represent the action
of the Chevalley generators on the tensor factors $\C^N$ (vertical) and $ M_\la$ (horizontal).
The following property of this action readily follows from the coproduct of the Chevalley generators:
suppose that  $\phi\in \Arr(v^m)$  and  $\phi\not \in \Arr(w_k)$. If $v^{r}=\phi v^m$, then
$\phi(w^m_k)=v^r_k$, i.e., the horizontal arrow yields the action of $\phi$ on the entire tensor product.
In general, $\phi(w^m_k)=v^r_k\mod \C v^m_s$, where $w_s=\phi w_k$.

Suppose that nodes of a column segment $BC$ (with $C$ the bottom node) belong to a $U_q(\g)$-submodule $M\subset \C^N\tp M_\la$. Let $\phi$ be a Chevalley generator assigned to
a horizontal arrow with the origin at this column. Consider the following situations:
\begin{enumerate}
\item The length of $\phi$ is $1$.
\begin{enumerate}
\item  There is no vertical $\phi$-arrow with the origin at $C$.
\item There is a vertical $\phi$-arrow with the origin at $C$.
\end{enumerate}
\item The length of $\phi$ is $2$, and the size of $BC$ is $2$ or greater. Let $C'$ and $C''$ be the nodes
$1$ and $2$ steps up, respectively.
\begin{enumerate}
\item  There is no vertical $\phi$-arrow with the origin at $C$ and at $C'$.
\item There is a vertical $\phi$-arrow with the origin either at $C$ or at $C'$.
\end{enumerate}
\end{enumerate}
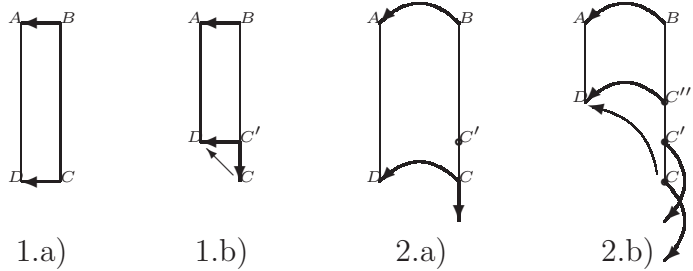
\begin{center}
\begin{picture}(40,95)
\put(15,30){\line(0,1){60}}\put(30,30){\line(0,1){60}}
\thicklines\put(30,30){\vector(-1,0){15}}

\thicklines\put(30,90){\vector(-1,0){15}}

\put(10,90){$\scriptscriptstyle{A}$}\put(30,90){$\scriptscriptstyle{B}$}
\put(30,30){$\scriptscriptstyle{C}$}\put(10,30){$\scriptscriptstyle{D}$}
\put(13 ,0){1.a)}
\end{picture}
\hspace{20pt}
\begin{picture}(40,95)
\put(15,45){\line(0,1){45}}\put(30,30){\line(0,1){60}}\put(15,45){\line(1,0){15}}

\thicklines\put(30,45){\vector(0,-1){15}}\thicklines\put(30,45){\vector(-1,0){15}}
\thicklines\put(30,90){\vector(-1,0){15}}


\put(10,90){$\scriptscriptstyle{A}$}\put(30,90){$\scriptscriptstyle{B}$}
\put(30,30){$\scriptscriptstyle{C}$}
\put(30,45){$\scriptscriptstyle{C'}$}\put(10,45){$\scriptscriptstyle{D}$}
\thinlines\put(27.5,32.5){\vector(-1,1){10}}
\put(13 ,0){1.b)}
\end{picture}
\hspace{20pt}
\begin{picture}(50,95)
\put(15,30){\line(0,1){60}}\put(45,30){\line(0,1){60}}

\thicklines\qbezier(15,90)(30,105)(45,90)\thicklines\put(18,93){\vector(-1,-1){3}}

\thicklines\qbezier(15,30)(30,45)(45,30)\thicklines\put(18,33){\vector(-1,-1){3}}
\thicklines\put(45,30){\vector(0,-1){15}}
\put(45,45){\circle{1.5}}

\put(10,90){$\scriptscriptstyle{A}$}\put(45,90){$\scriptscriptstyle{B}$}
\put(45,30){$\scriptscriptstyle{C}$}\put(45,45){$\scriptscriptstyle{C'}$}
\put(10,30){$\scriptscriptstyle{D}$}
\put(21 ,0){2.a)}
\end{picture}
\hspace{20pt}
\begin{picture}(50,95)
\put(15,60){\line(0,1){30}}\put(45,30){\line(0,1){60}}
\thinlines\multiput(45,30)(-6,6){6}{\line(-1,1){3}}
\thinlines\qbezier(42,33)(39,54)(18,58)\thicklines\put(19,58){\vector(-4,1){3}}

\thicklines\qbezier(15,90)(30,105)(45,90)\thicklines\put(18,93){\vector(-1,-1){3}}
\thicklines\qbezier(15,60)(30,75)(45,60)\thicklines\put(18,63){\vector(-1,-1){3}}

\thicklines\qbezier(45,45)(60,30)(45,15)\thicklines\put(48,18){\vector(-1,-1){3}}
\thicklines\qbezier(45,30)(60,15)(45,0)\thicklines\put(48,3){\vector(-1,-1){3}}

\put(45,30){\circle{1.5}}\put(45,45){\circle{1.5}}\put(45,60){\circle{1.5}}

\put(10,90){$\scriptscriptstyle{A}$}\put(45,90){$\scriptscriptstyle{B}$}
\put(45,30){$\scriptscriptstyle{C}$}\put(45,45){$\scriptscriptstyle{C'}$}
\put(45,60){$\scriptscriptstyle{C''}$}\put(10,60){$\scriptscriptstyle{D}$}
\put(21 ,0){2.b)}
\end{picture}
\end{center}
\begin{definition}
We call the transition from column $BC$ to column $AD$ an elementary  move or simply move of the length equal to the length
of $\phi$-arrow.
The elementary moves 1.a) and 2.a) are called left. The elementary moves 1.b) and 2.b) are called diagonal.
\end{definition}
\begin{lemma}[Elementary moves]
Under the conditions above, the column segment $AD$ lies in $M$.
\end{lemma}
\begin{proof}
Clear.
\end{proof}
We will use elementary moves to reach a node or collection of nodes in the diagram starting from the rightmost column, which
is assumed to be in a submodule $M$. That way we prove that the target nodes are in $M$.

Let $D'_\psi \subset D_\psi$ denote the subgraph whose nodes form the triangle lying above the principal diagonal, i.e.  $\{w_k^{m}\}_{k+m>i+j}$.
\begin{lemma}
\label{D'}
Suppose that $\psi=\psi^{ij}$ is a  principal monomial. Then  the linear span of $D'_\psi$ lies in $V_{j-1}$.
\end{lemma}
\begin{proof}
Suppose that all horizontal arrows in $D'_\psi$ have length $1$, as e.g.  for $\g=\s\o(2n+1)$, and $\g=\s\p(2n)$.
Consider the diagram $D_\psi$ on Fig.\ref{D_prin}.a, where
$D'_\psi$ is the triangle $ABC$. The column $BC$ belongs to $ V_{j-1}$ by construction. All arrows
have length $1$. Applying elementary diagonal moves we prove that $ABC$ is in $ V_{j-1}$.

Now suppose there is a horizontal arrows of length $2$.
Assuming $i\leqslant n-1$, $n'+1\leqslant j$, consider the diagram $D_\psi$  where the triangle $D'_\psi$  is denoted by $ABC$ (cf. Fig.\ref{D_prin}.b).
\begin{figure}[h]
\caption{}
\label{D_prin}
\begin{center}
\begin{picture}(150,160)
\put(150,0 ){\line(0,1){150}}\put(0,150){\line(1,0){150}}
\thinlines\put(0,150){\line(1,-1){150}}
\thinlines\put(15,150){\line(1,-1){135}}

\put(90,150){\line(0,-1){75}}\put(105,150){\line(0,-1){90}}

\thicklines\put(150,75){\vector(0,-1){15}}
\thicklines\put(105,150){\vector(-1,0){15}}
\thicklines\put(105,75){\vector(-1,0){15}}

\put(152,148){$\scriptscriptstyle{i}$}\put(152,0){$\scriptscriptstyle{j}$}
\put(148,152){$\scriptscriptstyle{j}$}\put(0,152){$\scriptscriptstyle{i}$}
\put(13,143){$\scriptscriptstyle{A}$}\put(83,143){$\scriptscriptstyle{D}$}\put(98,143){$\scriptscriptstyle{F}$}
\put(143,143){$\scriptscriptstyle{B}$}
\put(83,69){$\scriptscriptstyle{E}$}\put(106,69){$\scriptscriptstyle{G'}$}
\put(100,55){$\scriptscriptstyle{G}$}
\put(143,10){$\scriptscriptstyle{C}$}
\put(70,0){a)}
\end{picture}
\hspace{30pt}
\begin{picture}(150,160)
\put(150,0 ){\line(0,1){45}}\put(150,60 ){\line(0,1){90}}
\put(0,150){\line(1,0){90}}\put(105,150){\line(1,0){45}}
\thinlines\put(0,150){\line(1,-1){150}}
\thinlines\put(15,150){\line(1,-1){135}}
\put(150,30){\circle{1.5}}\put(150,45){\circle{1.5}}\put(150,60){\circle{1.5}}\put(150,75){\circle{1.5}}

\put(75,150){\line(0,-1){60}}\put(90,150){\line(0,-1){75}}\put(105,150){\line(0,-1){90}}\put(120,150){\line(0,-1){105}}

\put(150,30){\line(-1,1){30}}

\thicklines\put(150,45){\vector(0,-1){15}}\thicklines\put(150,75){\vector(0,-1){15}}
\thicklines\put(120,150){\vector(-1,0){15}}
\thicklines\put(90,150){\vector(-1,0){15}}
\thicklines\put(120,60){\vector(-1,0){15}}
\thicklines\put(150,30){\vector(-1,0){15}}

\thicklines\qbezier(150,75)(165,60)(150,45)\thicklines\qbezier(150,60)(165,45)(150,30)
\thicklines\put(153,33){\vector(-1,-1){3}}\thicklines\put(153,48){\vector(-1,-1){3}}

\thicklines\qbezier(90,150)(105,165)(120,150)\thicklines\put(93,153){\vector(-1,-1){3}}

\thicklines\put(90,90){\vector(-1,0){15}}
\thicklines\qbezier(90,75)(105,90)(120,75)\thicklines\put(93,78){\vector(-1,-1){3}}

\put(152,148){$\scriptscriptstyle{i}$}\put(152,0){$\scriptscriptstyle{j}$}
\put(148,152){$\scriptscriptstyle{j}$}\put(0,152){$\scriptscriptstyle{i}$}
\put(160,73){$\scriptscriptstyle{n-1}$}
\put(160,59){$\scriptscriptstyle{n}$}\put(160,44){$\scriptscriptstyle{n'}$}\put(160,29){$\scriptscriptstyle{n'+1}$}
\put(13,143){$\scriptscriptstyle{A}$}
\put(68,143){$\scriptscriptstyle{D}$}\put(83,143){$\scriptscriptstyle{F}$}\put(98,143){$\scriptscriptstyle{H}$}
\put(113,143){$\scriptscriptstyle{J}$}\put(143,143){$\scriptscriptstyle{B}$}
\put(68,85){$\scriptscriptstyle{E}$}\put(91,85){$\scriptscriptstyle{G'}$}
\put(83,70){$\scriptscriptstyle{G}$}\put(120,73){$\scriptscriptstyle{L''}$}
\put(100,55){$\scriptscriptstyle{I}$}\put(120,61){$\scriptscriptstyle{L'}$}
\put(113,40){$\scriptscriptstyle{L}$}
\put(143,25){$\scriptscriptstyle{M}$}
\put(143,10){$\scriptscriptstyle{C}$}
\put(70,0){b)}
\end{picture}
\end{center}
\end{figure}
The rightmost column $BC$  belongs to $ V_{j-1}$ by construction.
For each node in the trapezoid $JBCL$ there is a horizontal arrow of length $1$.
Those arrows are distinct from vertical arrows for all nodes in the line $L'M\subset JBCL$.
Apply the corresponding left moves  to the columns rested on $L'M$. This operation proves that
trapezoid $HBCI$ is in $ V_{j-1}$. Then apply the diagonal move of length $2$ to the column $JL$ and get $FG\subset  V_{j-1}$.
All arrows in the triangle $ADE$ have length $1$, therefore $ADE\subset  V_{j-1}$, via diagonal moves.

The case $i= n$, $n'+1\leqslant j$ dysplayed on Fig.\ref{D_prin1}.a is similar to already considered: all horizontal arrows within $D'_\psi$ are of length $1$.
The case $i\leqslant n-1$, $n'= j$ is displayed on Fig.\ref{D_prin1}.b:
\begin{figure}[H]
\caption{}
\label{D_prin1}
\begin{center}
\begin{picture}(90,90)
\put(35,80){\line(1,0){45}}
\thinlines\put(5,80){\line(1,-1){75}}
\thinlines\put(20,80){\line(1,-1){60}}
\put(80,80){\circle{1.5}}\put(80,65){\circle{1.5}}\put(80,50){\circle{1.5}}

\put(80,50){\line(0,-1){45}}

\thicklines\put(80,65){\vector(0,-1){15}}
\thicklines\put(35,80){\vector(-1,0){15}}

\thicklines\qbezier(80,80)(95,65)(80,50) \thicklines\put(83,53){\vector(-1,-1){3}}
\thicklines\qbezier(5,80)(20,95)(35,80)\thicklines\put(8,83){\vector(-1,-1){3}}

\put(83,78){$\scriptscriptstyle{i}$}\put(82,2){$\scriptscriptstyle{j}$}
\put(78,82){$\scriptscriptstyle{j}$}\put(3,82){$\scriptscriptstyle{i}$}
\put(18,73){$\scriptscriptstyle{A}$}
\put(73,73){$\scriptscriptstyle{B}$}
\put(72,15){$\scriptscriptstyle{C}$}
\put(40,0){a)}
\end{picture}
\hspace{30pt}
\begin{picture}(90,90)
\put(5,80){\line(1,0){45}}

\thinlines\put(5,80){\line(1,-1){75}}
\thinlines\put(20,80){\line(1,-1){60}}
\put(80,5){\circle{1.5}}\put(80,20){\circle{1.5}}\put(80,35){\circle{1.5}}

\put(50,80){\line(0,-1){30}}\put(65,80){\line(0,-1){45}}\put(80,80){\line(0,-1){45}}

\thicklines\put(80,35){\vector(0,-1){15}}
\thicklines\put(80,80){\vector(-1,0){15}}

\thicklines\put(80,35){\vector(-1,0){15}}

\thicklines\qbezier(80,35)(95,20)(80,5) \thicklines\put(83,8){\vector(-1,-1){3}}

\thicklines\qbezier(50,80)(65,95)(80,80)\thicklines\put(53,83){\vector(-1,-1){3}}
\thicklines\qbezier(50,50)(65,65)(80,50)\thicklines\put(53,53){\vector(-1,-1){3}}

\put(82,78){$\scriptscriptstyle{i}$}\put(82,2){$\scriptscriptstyle{j}$}
\put(78,82){$\scriptscriptstyle{j}$}\put(5,82){$\scriptscriptstyle{i}$}
\put(18,73){$\scriptscriptstyle{A}$}
\put(43,73){$\scriptscriptstyle{F}$}\put(58,73){$\scriptscriptstyle{D}$}\put(58,31){$\scriptscriptstyle{E}$}
\put(73,73){$\scriptscriptstyle{B}$}
\put(45,45){$\scriptscriptstyle{G}$}
\put(83,33){$\scriptscriptstyle{C'}$}\put(83,48){$\scriptscriptstyle{C''}$}

\put(72,15){$\scriptscriptstyle{C}$}
\put(40,0){b)}
\end{picture}
\end{center}
\end{figure}
Apply the diagonal
move of length $1$ to the column $BC'$ and get $DE\subset V_{j-1}$. Then apply the diagonal
move of length $2$ to $BC'$ and get $EG\subset V_{j-1}$. Thence the entire triangle $AFG$ is in $M$.
\end{proof}
\begin{propn}
\label{principal_diagram1}
Suppose $\psi \in \Psi_{\bt}$, $(i,j)\in P(\bt)$, and $\psi \not =\psi^{ij}$. Then $w_i\tp \psi v_\la \in  V_{j-1}$.
\end{propn}
\begin{proof}
Consider a factorization $\psi=\psi'\psi^{mj}$, where $m$ is some integer satisfying $i\prec m \preceq j$ and $\psi'\in \Psi_{\ve_i-\ve_m}$.  Choose $m$ to be the smallest possible. In the factorization $\psi^{ij}=\psi^{im}\psi^{mj}$ let $\phi$ be the rightmost Chevalley factor in $\psi^{im}$, while $\phi'$ the
rightmost factor in $\psi'$. Due to the choice of $m$, $\phi\not=\phi'$.
 Further we consider algebras of types $B,C$  separately from  $D$.

In diagrams of  types $B$ and $C$, all arrows have length $1$, Fig.\ref{D_nonprin}.a.
All nodes in the north-east rectangle $CDIH$ are the same as in $D_{\psi^{ij}}$. Therefore $CDGF$ is in $ V_{j-1}$, by Lemma \ref{D'}.
 Since $\phi'\not=\phi$, the left move via $\phi'$ maps $CF$ onto $BE$, modulo
$CF\subset  V_{j-1}$, proving $BE\subset  V_{j-1}$. Applying diagonal moves to $BE$ we get
the triangle $ADE \subset V_{j-1}$ including the node $A$, which is $w^i_i=w_i\tp \psi v_\la$.
\begin{figure}[H]
\caption{}
\label{D_nonprin}
\begin{center}
\begin{picture}(105,125)
\put(105,0 ){\line(0,1){105}}\put(0,105){\line(1,0){105}}
\thinlines\put(0,105){\line(1,-1){105}}
\thinlines\put(15,105){\line(1,-1){90}}

\put(45,105){\line(0,-1){45}}\put(60,105){\line(0,-1){60}}
\put(105,60){\line(-1,0){45}}\put(105,45){\line(-1,0){45}}

\thicklines\put(105,60){\vector(0,-1){15}}
\thicklines\put(60,105){\vector(-1,0){15}}
\thicklines\put(60,60){\vector(-1,0){15}}

\put(0,107){$\scriptscriptstyle{i}$}\put(107,103){$\scriptscriptstyle{i}$}
\put(107,0){$\scriptscriptstyle{j}$}\put(103,107){$\scriptscriptstyle{j}$}
\put(58,107){$\scriptscriptstyle{m}$}\put(107,43){$\scriptscriptstyle{m}$}
\put(50,111){$\scriptstyle{\phi'}$}\put(109,52){$\scriptstyle{\phi}$}
\put(-2,98){$\scriptscriptstyle{A}$}\put(38,98){$\scriptscriptstyle{B}$}\put(53,98){$\scriptscriptstyle{C}$}
\put(98,98){$\scriptscriptstyle{D}$}
\put(38,54){$\scriptscriptstyle{E}$}
\put(55,54){$\scriptscriptstyle{F}$}
\put(98,60){$\scriptscriptstyle{G}$}
\put(100,39){$\scriptscriptstyle{I}$}
\put(55,39){$\scriptscriptstyle{H}$}
\put(50,0){a)}
\end{picture}
\hspace{20pt}
\begin{picture}(105,125)
\put(105,0){\line(0,1){105}}\put(105,75){\line(0,1){30}}
\put(0,105){\line(1,0){30}}\put(60,105){\line(1,0){45}}
\thinlines\put(0,105){\line(1,-1){105}}
\thinlines\put(15,105){\line(1,-1){90}}
\put(60,105){\line(0,-1){60}}
\put(60,75){\circle{1.5}}

\put(90,105){\line(0,-1){90}}\put(75,105){\line(0,-1){75}}\put(30,105){\line(0,-1){30}}
\put(105,15){\line(-1,0){15}}\put(105,30){\line(-1,0){15}}

\thicklines\put(105,30){\vector(0,-1){15}}\thicklines\put(90,105){\vector(-1,0){15}}\thicklines\put(90,30){\vector(-1,0){15}}
\thicklines\qbezier(30,105)(45,120)(60,105)\thicklines\put(33,108){\vector(-1,-1){3}}
\thicklines\qbezier(30,75)(45,90)(60,75)\thicklines\put(33,78){\vector(-1,-1){3}}

\thicklines\qbezier(105,75)(120,60)(105,45)\thicklines\put(108,48){\vector(-1,-1){3}}
\thicklines\qbezier(60,75)(75,60)(60,45)\thicklines\put(63,48){\vector(-1,-1){3}}

\put(0,107){$\scriptscriptstyle{i}$}\put(107,103){$\scriptscriptstyle{i}$}
\put(107,0){$\scriptscriptstyle{j}$}\put(103,107){$\scriptscriptstyle{j}$}
\put(88,107){$\scriptscriptstyle{m}$}\put(107,13){$\scriptscriptstyle{m}$}
\put(80,111){$\scriptstyle{\phi'}$}\put(109,23){$\scriptstyle{\phi}$}
\put(-2,98){$\scriptscriptstyle{A}$}\put(68,98){$\scriptscriptstyle{B}$}\put(83,98){$\scriptscriptstyle{C}$}
\put(98,98){$\scriptscriptstyle{D}$}
\put(68,24){$\scriptscriptstyle{E}$}
\put(83,24){$\scriptscriptstyle{F}$}
\put(98,31){$\scriptscriptstyle{G}$}
\put(100,10){$\scriptscriptstyle{I}$}
\put(83,10){$\scriptscriptstyle{H}$}
\put(50,0){b)}
\end{picture}
\hspace{20pt}
\begin{picture}(105,125)
\put(105,0){\line(0,1){105}}
\put(0,105){\line(1,0){45}}\put(75,105){\line(1,0){30}}
\thinlines\put(0,105){\line(1,-1){105}}
\thinlines\put(15,105){\line(1,-1){90}}

\put(75,105){\line(0,-1){75}}\put(45,105){\line(0,-1){45}}
\put(105,45){\line(-1,0){30}}\put(105,30){\line(-1,0){30}}

\thicklines\put(105,45){\vector(0,-1){15}}
\thicklines\qbezier(45,105)(60,120)(75,105)\thicklines\put(48,108){\vector(-1,-1){3}}
\thicklines\qbezier(45,60)(60,75)(75,60)\thicklines\put(48,63){\vector(-1,-1){3}}

\put(0,107){$\scriptscriptstyle{i}$}\put(107,103){$\scriptscriptstyle{i}$}
\put(107,0){$\scriptscriptstyle{j}$}\put(103,107){$\scriptscriptstyle{j}$}
\put(73,107){$\scriptscriptstyle{m}$}\put(107,28){$\scriptscriptstyle{m}$}
\put(60,116){$\scriptstyle{\phi'}$}\put(109,38){$\scriptstyle{\phi}$}
\put(-2,98){$\scriptscriptstyle{A}$}\put(38,98){$\scriptscriptstyle{B}$}\put(68,98){$\scriptscriptstyle{C}$}
\put(98,98){$\scriptscriptstyle{D}$}
\put(38,54){$\scriptscriptstyle{E}$}
\put(75,58){$\scriptscriptstyle{F'}$}\put(68,39){$\scriptscriptstyle{F}$}
\put(98,60){$\scriptscriptstyle{G}$}
\put(100,39){$\scriptscriptstyle{I}$}
\put(68,25){$\scriptscriptstyle{H}$}
\put(50,0){c)}
\end{picture}
\end{center}
\end{figure}
Now we look at the type $D$.  We can assume that $i\leqslant n-1,n'+1\leqslant j$, since otherwise
this case reduces to already considered. If the length of $\phi'$ is $1$, the reasoning is the same as
above. The only difference is that one may have to use a diagonal move of length $2$ in transition from $BE$ to $A$, see
Fig.\ref{D_nonprin}.b.  If the length of $\phi'$ is  $2$, then the transition to $BE$ is performed via $\phi'$ applied to
 $CF'\subset  V_{j-1}$, as shown
on Fig.\ref{D_nonprin}.c. This proves that $BE \subset  V_{j-1}$.
Further, all horizontal arrows in the triangle $ABE$ are of length $1$ (the factor $f_{\al_n}$ enters $\psi$ only once).
This situation is similar to the types $B$ and $C$ considered earlier. Thus, the node
$A=w^i_i=w_i\tp \psi v_\la$,  belongs to $ V_{j-1}$.
\end{proof}
For $i\preccurlyeq j$ denote by $|\!\!|i-j|\!\!|$  the distance (the number of arrows in a path) from $i$ to $j$ on the Hasse diagram
of the natural representation of $U_q(\g_-)$.
\begin{propn}
\label{principal_diag}
Suppose that $i,j\in I$ are such that $i\prec j$. Then
\be
w_i\tp \psi^{ij}v_\la &=&(-1)^{|\!\!|i-j|\!\!|} q^{-(\la,\vt_{ij})}w_{j}\tp v_\la \mod  V_{j-1}.
\ee
\end{propn}
\begin{proof}
Suppose that $\al\in \Pi^+$ and $(i,k)\in P(\al)$. By Lemma \ref{D'}, the node $w_{i}\tp \psi^{kj}v_\la\in D'_\psi$ lies
in $V_{j-1}$
Applying $\Delta f_\al= f_\al\tp q^{-h_\al}+1\tp f_\al$ to $w_{i}\tp \psi^{kj}v_\la$ we get
\be
w_i\tp \psi^{ij}v_\la &=&q^{-(\la,\al)-(\al,\ve_j-\ve_k)}w_{k}\tp \psi^{kj}v_\la =q^{-(\la,\vt_{ij}-\vt_{kj})}w_{j}\tp \psi^{kj}v_\la \mod  V_{j-1}
\nn
\ee
for all $k\preccurlyeq j$. Here we used $f_\al w_i=w_k$ and $f_\al\psi^{kj}=\psi^{ij}$ for all $k\preccurlyeq j$.
Proceeding recursively along the path from $i$ to $j$ with the boundary condition $\vt_{jj}=0$ we complete the proof.
\end{proof}
\subsection{Generalized parabolic Verma modules}
Fix a generalized Levi subalgebra $\k\subset \g$ and a weight $\la\in \Cg^*_{\k,reg}$.
Let $M^\h_\la$ denote the Verma module of highest weight $\la$.
For each $\al \in \Pi^+_\k$, there is  a singular vector $v_{\la-\al} \in M^\h_\la$ generating a submodule $M^\h_{\la-\al}\subset M^\h_{\la}$, cf. Section \ref{Sec:RedShapInf}.  Set  $M^\k_\la$ to be the quotient of $M^\h_\la$ by
the submodule $\sum_{\al \in \Pi^+_\k}M^\h_{\la-\al}$.

We denote by  $V^\k_\bullet=(V^\k_{i})_{i=1}^N$ the filtration of $\C^N\tp M^\k_\la$ by the  modules  $V^\k_i$  generated by
$w_k\tp v_\la$, $k=1,\ldots,i$. For $\k=\h$ it is the standard filtration considered in the previous sections. Clearly
$V^\k_\bullet$ is obtained from $V^\h_\bullet$ through the projection $\C^N\tp M^\h_\la\to \C^N\tp M^\k_\la$.
Further we show that $V^\k_j/V^\k_{j-1}$ vanishes once $j\in \bar I_\k$ and $q$ is close to $1$.

\begin{propn}
For each $\la\in \Cg^*_{\k,reg}$ there is a neighborhood $\Omega$ of $1$ in $\C$ such that
the submodule $V^\k_j$ is generated by $w_i\tp v_\la$, $i\leqslant j$, $i\in I_\k$, for all $q\in \Omega$.
\end{propn}
\begin{proof}
For all $j$ denote by $V_j'\subset V^\k_j$ the submodule generated by all $w_i\tp v_\la$ with $i\leqslant j$ and $i\in I_\k$.
We aim to prove that $V_j'=V^\k_j$.

The statement is trivial for $j=1$. Suppose  it is true for all $i<j$. If $j\in I_\k$, then $V^\k_j$ is generated
by $w_j\tp v_\la$ and by $V^\k_{j-1}=V_{j-1}'$, hence the proof.
Suppose that $j\in \bar I_\k$.
Choose the greatest $i$ such that $i\lessdot j$. Then $(i,j)\in P(\al)$ for some $\al \in \Pi^+_\k$.
By Lemma \ref{dynamical_roots_deformed} there exists an open set $\Omega\subset \C$ containing $1$ such that
the principal term in  $\check{f}_{ij}v_\la \simeq v_{\la-\al}$ is not zero for all $q\in \Omega$.
Then $w_j\tp v_\la\simeq w_i\tp \psi^{ij}v_\la\simeq w_i\tp \check{f}_{ij}v_\la =0$
modulo $V^\k_{j-1}$, by Propositions \ref{principal_diag} and \ref{principal_diagram1}.
 By the induction assumption, we conclude that $w_j\tp v_\la\in V_{j-1}'$ and $V^\k_j=V_{j-1}'=V_j'$.
\end{proof}
\noindent
\begin{corollary}
The graded module $\gr V^\k_\bullet$ is isomorphic to the direct sum $\op_{j\in I_\k}V^\k_j/V^\k_{j-1}$.
\label{standard k-filtration}
\end{corollary}
Recall that the tensor  $\Ru_{21}\Ru$ commutes with $\Delta(x)$ for all $x\in U_q(\g)$, \cite{D}.
\begin{propn}
The invariant operator  $\Q=(\pi\tp \id)(\Ru_{21}\Ru)$ preserves the standard filtration. It is
scalar on each graded component $V^\k_j/V^\k_{j-1}$, $j\in I_\k$, with the eigenvalue
\be
x_j=q^{2(\la+\rho,\ve_j)-2(\rho,\ve_1)+||\ve_j||^2-||\ve_1||^2},
\label{eigenvalues}
\ee
unless $V^\k_j/V^\k_{j-1}\not =\{0\}$.
\end{propn}
\begin{proof}
The operator $\Q$ can be presented as $\Delta(z)(z^{-1}\tp z^{-1})$, for a certain central element $z$, \cite{D1}.
Therefore $\Q$ is a scalar multiple on every submodule and factor module of highest weight of $\hat V_N$.
Now we do induction on $j$. The submodule $V^\h_1$ is of highest
weight, thence it is $\Q$-invariant. Suppose that $V^\h_{j-1}$ is $\Q$-invariant for  $j>1$. Since $\Q$ is scalar on $V_j^\h/ V^\h_{j-1}$, the submodule $V^\h_j$ is $\Q$-invariant.

The eigenvalue of $\Q$ on $V^\h_j/V^\h_{j-1}$ is determined by its highest weight and equal to (\ref{eigenvalues}), for all $j\in I$,
\cite{M1}.
So the proposition is proved for $\k=\h$. The general case is obtained from this by taking projection to $\C^N\tp M^\k_\la$
and applying Corollary \ref{standard k-filtration}.
\end{proof}
It follows that $\Q$ satisfies the polynomial equation $\prod_{j\in I_\k}(\Q-x_j)=0$ on $\C^N\tp M^\k_\la$.
We will not address the issue if  $V^\k_j/V^\k_{j-1}$ survive for all $j\in I_\k$ as we bypass it in what follows.

\section{Representations of quantum conjugacy classes}
In this section we extend the ground field $\C$ to the local ring  $\C[\![\hbar]\!]$ of formal power series in $\hbar$.
 The quantum group $U_\hbar(\g)$ is a completion
of the $\C[q,q^{-1}]$-algebra $U_q(\g)$ in the $\hbar$-adic topology via the extension $q=e^\hbar$. Its Cartan subalgebra $U_\hbar(\h)$ can be  generated by  $h_\al \in \h$ instead of $q^{\pm h_\al}$.

Assuming that $\k$ is fixed, we suppress the corresponding superscripts and write simply  $M_\la=M^\k_\la$ and $V_\bullet=V^\k_\bullet$.
\begin{propn}
\label{M_la is free}
Suppose that $\la\in \Cg^*_{\k,reg}$. Then $M_{\la}$ is $\C[\![\hbar]\!]$-free.
\end{propn}
\begin{proof}
The proof is similar to \cite{M3}, Proposition 6.2, where it is done for a regular  pseudo-parabolic Verma module over
$U_q\bigl(\s\p(n)\bigr)$.
The crucial observation is that for all $\al\in \Pi^+_\k$ and $\la\in \Cg^*_{\k,reg}$ the vectors
$\hat{f}_{ij}(\la)$ with $(i,j)\in P(\al)$ can
be included in a PBW basis in $U_\hbar(\g_-)$ if the ring of scalars is $\C[\![\hbar]\!]$. This follows from Lemma \ref{dynamical_roots_deformed}.
\end{proof}
\noindent
Proposition \ref{M_la is free} implies that the algebra $\End(M_\la)$ is also $\C[\![\hbar]\!]$-free.
We are going to realize a quantized conjugacy class of the point
$\lim_{\hbar \to 0}q^{2h_\la}\in T^\k_{reg}$ as a subalgebra in $\End(M_\la)$.

Consider the image of the algebra $\C_{\hbar}[G]$ in $\End(M_\la)$ under the composition homomorphism
$$
\C_{\hbar}[G]\to U_q(\g)\to \End(M_\la).
$$
Here the algebra $U_q(\g)$ is extended over $\C[\![\hbar]\!]$.
This representation induces a character, $\chi_\la$,  of the center of $\C_{\hbar}[G]$.
It annihilates the ideal in $\C_\hbar[G]$ generated by the kernel $\chi_\la$
and by the entries of the minimal polynomial of $\Q$ as a linear operator on $\C^N\tp M_\la$.
The center of $\C_\hbar[G]$ is generated by
\be
\label{q-trace}
\tau_k&=&\Tr_{q}(\Q^k):=\Tr\bigl((\pi(q^{2h_\rho})\tp 1)\Q^k\bigr)\in U_\hbar(\g), \quad k=1,2,\ldots,
\nn\\
\tau^-&=&\Tr_q(\Q_+)-\Tr_q(\Q_-), \quad \mbox{for}\quad \g=\s\o(2n).
\nn
\ee
Here $\Q_\pm$ are the images of $\Ru_{21}\Ru$ in $\End(W_\pm)\tp U_q(\g)$, were
$W_\pm\subset \wedge^n(\C^n) $ are finite dimensional irreducible modules of highest weights
$\sum_{i=1}^{n-1} \ve_i\pm\ve_n$.
In the classical limit, this invariant separates two $SO(2n)$-conjugacy classes whose eigenvalues are all distinct from $\pm 1$.
They  are flipped by any inversion $x_i\leftrightarrow x_i^{-1}$, $i=1,\ldots,n $,
and amount to an $O(2n)$-conjugacy class. If $\pm 1$ is in the spectrum, the $O(2n)$-conjugacy class is also an $SO(2n)$-class. In this case, $\tau^-$ is redundant.

\begin{thm}
Let $\k\subset \g$ be a generalized Levi subalgebra,  $\la \in \Cg^*_{\k,reg}$, and
$M_\la=M^\k_\la$  the corresponding generalized parabolic Verma module. Then\\
i) the annihilator of $M_\la$ in $\C_\hbar[G]$ is generated by
$$
\begin{array}{clll}
\hspace{50pt}{\displaystyle\Bigl(\prod_{i\in I_\k}}(\Q-x_i)\Bigr)_{ij},& i,j=1,\ldots,N,
\nn\\[10pt]
\chi_\la(\tau_k)
-
{\displaystyle\sum_{i=1}^N }x_i^k
{\displaystyle\prod_{\al\in \Rm_+}\frac{ q^{(\la+\rho+\ve_i,\al)}-q^{-(\la+\rho+\ve_i,\al)}}{ q^{(\la+\rho,\al)}-q^{-(\la+\rho,\al)}}}
,&k=1,\ldots,N,
\nn\\[10pt]
\chi_\la(\tau^{-})
-
{\displaystyle\prod_{i=1}^n}\bigr(q^{2(\la+\rho,\ve_i)} -q^{-2(\la+\rho,\ve_i)} \bigl),
&\g=\s\o(2n),
\nn
\end{array}
$$
where $x_i$ is given by (\ref{eigenvalues}),\\
ii) the image of $\C_\hbar[G]$ in $\End(M_\la)$ is an equivariant quantization of $\C_\hbar[O_x]$, $x=\lim_{\hbar \to 0} q^{2h_\la}$,\\
iii) this quantization is independent of the choice of initial point and is an exact representation of the unique quantum conjugacy class
of $x$.
\end{thm}
\begin{proof}
The statements i) and ii) for all types of classes are proved in \cite{M1,M3,M4,AM1}, for certain regular $\k=\k_0$.
For arbitrary $\k$ there is an element $\si$ of the Weyl group such that $\Rm^+_\k=\si(\Rm^+_{\k_0})$.
The shifted action $\la_0\mapsto  \si(\la_0+\rho)-\rho=\la $ takes $\C^*_{\k_0,reg}$ to
$\C^*_{\k,reg}$. It  preserves the central characters and takes the set of  eigenvalues of $\Q$ on $\C^N\tp \hat M_{\la_0}^{\k_0}$ to eigenvalues on $\C^N\tp \hat M_{\la}^\k$.
Moreover, $\si\{x_i\}_{i\in I^{\k_0}}=\{x_i\}_{i\in I^{\k}}$ as $\si$ relates the orderings $\lessdot$ relative to $\k_0$ and $\k$. This implies that the annihilator of $M_\la^{\k_0}$ in $\C_\hbar[G]$
vanishes on  $M_\la^\k$, that is, there is an equivariant homomorphism $\C_\hbar[G/K_0]\to \End(M_\la^\k)$.
In order to complete the proof, we need to show that this homomorphism is an embedding.

Since $\C_\hbar[G/K_0]$ is a direct sum of $\C[\![\hbar]\!]$-finite isotypic $U_\hbar(\g)$-components and $\End(M_\la^\k)$ is $\C[\![\hbar]\!]$-free, the image of $\C_\hbar[G/K_0]$ is $\C[\![\hbar]\!]$-free. The algebra $\C[G/K_0]$ has no proper invariant ideals, hence the kernel
of the map $\C_\hbar[G/K_0]\to \End(M_\la^\k)$ is zero. This completes the proof.
\end{proof}
\vspace{1cm}
{\bf Acknowledgements}.
This research is supported in part by RFBR grant 15-01-03148.
The author is grateful to the Max-Planck Institute for Mathematics, Bonn for hospitatlity.

\end{document}